\documentclass{amsart}

\usepackage{amssymb,amsmath,amsthm, color}
\usepackage[margin=2.8 cm]{geometry}
\usepackage{url}
\usepackage[utf8]{inputenc}
\usepackage{amsfonts,amssymb,amsmath,amsthm,bbm,systeme}
\usepackage{url}
\usepackage{enumerate}

\usepackage{amsmath,amsfonts,amsthm,url,color,amssymb, mathtools}
\usepackage{graphicx}
\usepackage[all]{xy}\usepackage[font=small,labelfont=bf]{caption}
\usepackage{tikz-cd} 

\usetikzlibrary{cd}
\usetikzlibrary{arrows}

\newcommand{\F}{\mathbb{F}}
\newcommand{\fq}{\mathbb{F}_q}
\newcommand{\fp}{\mathbb{F}_p}
\newcommand{\fqn}{\mathbb{F}_{q^n}}

\DeclareMathOperator{\tr}{Tr}
\DeclareMathOperator{\lcm}{lcm}
\numberwithin{equation}{section}

\newtheorem{theorem}{Theorem}[section]
\newtheorem{lemma}[theorem]{Lemma}

\newtheorem{corollary}[theorem]{Corollary}

\newtheorem{definition}[theorem]{Definition}

\usepackage{amssymb}
\newtheorem{remark}[theorem]{Remark}
\usepackage{enumerate}
\usepackage{graphicx}
\usepackage{algpseudocode, algorithm}
\usepackage{hyperref}

\begin{document}
\title{Existence of normal elements with prescribed norms}



\author{Arthur Fernandes}
\address{Departamento de Matem\'{a}tica,
Universidade Federal de Minas Gerais,
UFMG,
Belo Horizonte MG (Brazil),
 31270901}
\email{arthurfpapa@gmail.com}

\author{Daniel Panario}
\address{School of Mathematics and Statistics, Carleton University, 1125 Colonel By Dr., Ottawa, ON, K1S5B6, Canada}
\email{daniel@math.carleton.ca}

\author{Lucas Reis}
\address{Departamento de Matem\'{a}tica,
Universidade Federal de Minas Gerais,
UFMG,
Belo Horizonte MG (Brazil),
 31270901}
\email{lucasreismat@mat.ufmg.br}



\date{\today}
\begin{abstract}
For each positive integer $n$, let $\F_{q^n}$ be the unique $n$-degree extension of the finite field $\F_q$ with $q$ elements, where $q$ is a prime power. It is known that for arbitrary $q$ and $n$, there exists an element $\beta\in \F_{q^n}$ such that its Galois conjugates $\beta, \beta^q, \ldots, \beta^{q^{n-1}}$ form a basis for $\F_{q^n}$ as an $\F_q$-vector space. These elements are called normal and they work as additive generators of finite fields. On the other hand, the multiplicative group $\F_{q^n}^*$ is cyclic and any generator of this group is a primitive element. Many past works have dealt with the existence of primitive and normal elements with specified properties, including the existence of primitive elements whose traces over intermediate extensions are prescribed. Inspired by the latter, in this paper we explore the existence of normal elements whose norms over intermediate extensions are prescribed. We combine combinatorial and number-theoretic ideas and obtain both asymptotic and concrete results. In particular, we completely solve the problem in the case where only one intermediate extension is considered.

\end{abstract}

\maketitle

\section{Introduction}
Let $q$ be a prime power and let $\F_q$ be the finite field with $q$ elements. Given a positive integer $n$, the unique $n$-degree extension $\F_{q^n}$ of $\F_q$ has two remarkable structures related to the field operations. First, the multiplicative group $\F_{q^n}^*$ is cyclic and any generator of such group is a {\em primitive element}.  These elements are relevant in applications such as design theory, coding theory and cryptography (most notably, the Diffie-Hellman key exchange~\cite{dh}). On the other hand, the field $\F_{q^n}$ also has an $\F_q[x]$-module structure induced by the Frobenius map $\sigma(\alpha)\mapsto \alpha^q$. Namely, for $f(x)=\sum_{i=0}^ma_ix^i\in \F_q[x]$ and $\alpha\in \F_{q^n}$, we set 
$$f(x)\circ \alpha=\sum_{i=0}^{n}a_i\alpha^{q^i}.$$
It is known that $\F_{q^n}$ is cyclic with respect to this $\F_q[x]$-module structure, i.e., there exists $\beta\in \F_{q^n}$ such that any $\alpha \in \F_{q^n}$ is written as $f_{\alpha}(x)\circ \beta $ for some $f_{\alpha}\in \F_q[x]$. The latter is equivalent to $\{\beta, \beta^q, \ldots, \beta^{q^{n-1}}\}$ being a basis for the field $\F_{q^n}$ regarded as an $\F_q$-vector space. Any such $\beta$ is a {\em normal element} over $\F_q$. Normal elements are also of great importance in applications, being particularly relevant in those where field arithmetic must be performed. For more details on normal elements, see~\cite{GAO}.

It is well known that both primitive and normal elements exist for any finite extension of finite fields. Beyond the relevance of these  elements, we can see that they work as "generators" of $\F_{q^n}$ with respect to the two aforementioned cyclic structures of the field. In 1987, Lenstra and Schoof proved that, for every prime power $q$ and every positive integer $n$, there exists an element $\alpha\in \F_{q^n}$ that is simultaneously primitive and normal over $\F_q$. The proof is mainly based on what is known nowadays as the {\em character sum method}. The method basically works as follows. Suppose that we want to prove that two sets $A, B\subseteq \F_{q^n}$ (e.g., set of primitive and normal elements) have nontrivial intersection.  If one writes the indicator functions $1_A$ and $1_B$ of $A$ and $B$, respectively, by means of additive/multiplicative characters of finite fields, analytical arguments can derive that one of the following sums
$$\sum_{x\in \F_{q^n}}1_A(x)\cdot 1_B(x)=\sum_{x\in A}1_B(x)=\sum_{x\in B}1_A(x),$$
is positive. This method has been widely used to prove existence theorems; classical results like bounds for Gauss sums and Weil's bound are frequently employed. For more details, see~\cite{char} and the references therein.

In 1990, Cohen~\cite{C1} explored the existence of a primitive element $\alpha \in \F_{q^n}^*$ whose trace over $\F_q$ is prescribed, i.e., 
$$\sum_{i=0}^{n-1}\alpha^{q^{i}}=a,$$
where $a\in \F_q$ is arbitrarily chosen. He proved that we can choose $a$ generically unless $a=0$ and either $n=2$ or $q=4$ and $n=3$. Moreover, these are genuine exceptions to the problem. In 2022, inspired by Cohen's work, in~\cite{rr} the authors explored a natural extension of this issue, namely the existence of primitive elements whose traces over several intermediate extensions are prescribed. They obtained both asymptotic and concrete results, including an existence theorem where the character sum method is not even employed (this is a byproduct of Cohen's result). For more details, see Section 2 of~\cite{rr}. More recently~\cite{R24}, the problem of primitive elements with prescribed traces provided that $n$ is fixed and $q$ is large enough has been completely solved.

We observe that the trace map has a ``multiplicative twin", namely the norm map $\alpha\mapsto \alpha^{q^{n-1}+\cdots+q+1}$. This gives rise to a natural additive-multiplicative analogue of the issues discussed in~\cite{C1,rr} and this is the subject of study in this paper. More specifically, we discuss the existence of elements $\beta\in \F_{q^n}$ that are normal over $\F_q$ and whose norms over intermediate extensions are prescribed. 

To this end, we briefly comment on the machinery and results obtained, along with the structure of the paper. First, we obtain an analogue of Cohen's Theorem.

\begin{theorem}\label{thm1}
Let $n\ge 2$ and $a\in\fq^*$. Then there exists a normal element $\alpha\in\fqn$ such that $\alpha^{q^{n-1}+\cdots+q+1}=a$ with the sole genuine exception $q=3, n=2$ and $a=-1$. 
\end{theorem}
We observe that the condition $a\ne 0$ is really necessary: if $a=0$, then $\alpha=0$ and this element is clearly not normal over $\F_q$.  Next, we explore the existence of normal elements with several prescribed norms over intermediate fields. In a similar fashion to~\cite{rr}, the norm of elements (normal or not) cannot be prescribed freely since the field norms are not completely independent from each other. In fact, we require a ``gluing condition" which is reminiscent of the Chinese Remainder Theorem; see Lemma~\ref{lema1}. In the case where these conditions are fulfilled, we employ the character sum method to give sufficient conditions on the existence of normal elements with several prescribed norms; see Theorem~\ref{thm2}. The latter provides asymptotic results, covering a large class of families of intermediate extensions.
As in~\cite{rr}, we also obtain an existence theorem which is simply a byproduct of Theorem~\ref{thm1} and does not require further character sum estimates;  see~Theorem~\ref{thm3}.
Finally, we restrict ourselves to the case where the degrees of the intermediate extensions are pairwise relatively prime and obtain concrete results; see Theorem~\ref{thm4}.

The paper is organized as follows. In Section 2, we introduce some relevant notation and present our main results. In Section 3, we provide some background machinery, including bounds on certain character sums. In Section 4, we prove our results.

\section{Main Results}
The norm of an element in $\fqn$ over intermediate $\F_q$-extensions plays a central role in this paper and, for convenience, we adopt the following compact notation.

\begin{definition}
For $n > 1$, $d$ a divisor of $n$ and $\alpha \in \mathbb{F}_{q^n}^*$, we set
$$N_{n/d}(\alpha)=\alpha^{q^{n-d}+\cdots+q^d+1}=\alpha^{\frac{q^n-1}{q^d-1}}$$
the norm of $\alpha$ over $\mathbb{F}_{q^d}$.
\end{definition}


 In order to state our general result on the existence of normal elements with several prescribed norms, we must impose a condition to the problem. It is well known that the norm function is transitive, that is, if $e$ divides $d$ and $d$ divides $n$, then we have $$N_{n/{e}}(\alpha)=N_{n/e}(N_{d/{e}}(\alpha))$$ for every $\alpha$ in $\fqn$. 
In particular, if $d_1 < \cdots < d_k < n$ are
divisors of $n$ and $a_i \in \mathbb{F}_{q^{d_i}}$, $1 \leq i \leq k$, in order to have an element $\alpha \in \mathbb{F}_{q^n}$
with $N_{n/d_i}(\alpha) = a_i$ we necessarily need that
\begin{equation}\label{eq1}
    N_{d_i/\gcd(d_i,d_j)}(a_i) = N_{n/\gcd(d_i,d_j)}(\alpha) = {N}_{d_j/\gcd(d_i,d_j)}(a_j), \quad 1 \leq i, j \leq k.
\end{equation}

This condition is also sufficient, as we show in Lemma \ref{lema1}. From Eq.~\eqref{eq1} we have that if $d_i$ divides some $d_j$, then $N_{n/d_i}(\alpha) = a_i$ is already implied by $N_{n/d_j}(\alpha) = a_j$. Therefore, we focus only on the divisors $d_1 < \cdots < d_k < n$ of $n$ such that $d_i \nmid d_j$ for any $1 \leq i < j \leq k$. For the case $k > 1$, that is, $n$ is not a prime power, the following notation is useful.

\begin{definition}
Let $n > 1$ be an integer that is not a prime power, let $\delta(n)$ be the number of positive divisors of $n$ and fix  $1 < k < \delta(n)$.
\begin{enumerate}[(i)]
    \item $\Gamma_k(n)$ stands for the set of tuples $\mathcal{D}=(d_1, \ldots, d_k)$ of divisors of $n$ such that $d_1 < \cdots < d_k < n$ and 
     $d_i$ does not divide $d_j$ whenever $i \neq j$.
    \item For $\mathcal{D} = (d_1, \ldots, d_k) \in \Gamma_k(n)$ set $\mathbb{F}^*_\mathcal{D} = \prod_{i=1}^{k} \mathbb{F}_{q^{d_i}}^*$. 
    
    \item A tuple $\mathcal{A} = (a_1, \ldots, a_k) \in \F_{\mathcal{D}}^*$ is $\mathcal{D}$-admissible if, for every $1 \leq i < j \leq k$,
    \[
    N_{d_i/\gcd(d_i, d_j)}(a_i) = N_{d_j/\gcd(d_i, d_j)}(a_j).
    \]
\end{enumerate}
\end{definition}
From previous observation, the norms of elements over intermediate extensions must come from $\mathcal{D}$-admissible tuples. Our main result on the existence of normal elements with prescribed norms can be stated as follows.

\begin{theorem}\label{thm2}
Let $n > 1$ be an integer that is not a prime power, $1 < k < \delta(n)$, $\mathcal{D} =
(d_1, \ldots, d_k) \in \Gamma_k(n)$ and let $\mathcal{A} = (a_1, \ldots, a_k) \in \F^*_\mathcal{D}$ be a $\mathcal{D}$-admissible tuple. Then there
exists a normal element $\alpha \in \fqn^*$ with $N_{n/d_i}(\alpha) = a_i$ for every $1 \leq i \leq k$ provided that
\begin{equation}\label{eq2}
\gcd\left(\frac{q^n-1}{q^{d_1}-1}, \ldots, \frac{q^n-1}{q^{d_k}-1}\right)\geq W_q(x^n-1)q^{n/2}.
\end{equation}
Here, $W_q(f)$ denotes the number of monic squarefree divisors of $f$ over $\fq$.

\end{theorem}

We also obtain the following result.

\begin{theorem}\label{thm3}
Theorem \ref{thm2} also holds if Ineq.~\eqref{eq2} is replaced by the inequality $$\lcm(d_1,
\ldots,d_k)<n.$$ 
\end{theorem}

While Theorem \ref{thm2} is quite general and yields asymptotic results, it can be ineffective in some cases. For instance, if the GCD in Ineq.~\eqref{eq2} is smaller than $q^{n/2}$ (this occurs when $d_k=n/2$), then such inequality is never true. However,  under the restriction $\gcd(d_i, d_j)=1$ for every $1\leq i\neq j\leq k$, we obtain the following concrete result. 

\begin{theorem}\label{thm4}
    Let $n > 1$ be an integer that is not a prime power, $1 < k < \delta(n)$, $\mathcal{D} =
(d_1, \ldots, d_k) \in \Gamma_k(n)$ with $\gcd(d_i,d_j)=1$ and $n=d_1\cdots\ d_k$, and let $\mathcal{A} = (a_1, \ldots, a_k) \in \F^*_\mathcal{D}$ be a $\mathcal{D}$-admissible tuple. Then, there
exists a normal element $\alpha \in \fqn^*$ with prescribed norms $N_{n/d_i}(\alpha) = a_i$ for every $1 \leq i \leq k$ provided that
\begin{enumerate}
    \item $k\geq 3$,
    \item $k=2, d_1\geq 7$ and $q\ge 2$ or $d_1\in \{3, 4, 5, 6\}$ with $(d_1, d_2)\ne (3, 4)$ and $q$ is large enough. 
\end{enumerate}
\end{theorem}

We end this section with some important observation.

\begin{remark}
\begin{enumerate}
    \item The condition $n=d_1\cdots d_k$ in Theorem~\ref{thm4} can be dropped. In fact, since the $d_i$'s are pairwise relatively prime, we get $\lcm(d_1, \ldots, d_k)=d_1\cdots d_k$ and then the case $d_1\cdots d_k<n$ follows directly by Theorem~\ref{thm3}.    

    \item Theorem~\ref{thm4} is obtained after a careful analysis of the expressions appearing in Ineq.~\eqref{eq2}. In particular, we can obtain explicit constants for the asymptotic statement in item 2 of Theorem~\ref{thm4}. However, these constants are too large to complete the result using computer search so we omit them. 

\item The condition $\gcd(d_i, d_j)=1$ in Theorem~\ref{thm4} is not restrictive if $k=2$. In fact, if $\gcd(d_1, d_2)=d$ and $Q=q^d$, then $\F_{q^{d_i}}=\F_{Q^{d_i'}}$ with $d_i'=d_i/d$. Therefore,  $\gcd(d_1', d_2')=1$. 
\end{enumerate}

\end{remark}

\section{Preliminaries}
In this section we provide some background machinery that is further used in the proof of our results. We start with the following lemma that might be of independent interest. It provides a general result on the existence and number of elements in $\F_{q^n}$ whose norms over some intermediate extensions are prescribed.

\begin{lemma}\label{lema1}
    Let $n$ be a positive integer and $d_1<\cdots<d_k$ be divisors of $n$. 
    If $a_i\in\F_{q^{d_i}}^*, 1\leq i\leq k$, there exists an element $\alpha\in\fqn^*$ with $N_{n/d_i}(\alpha)=a_i$ if 
    $$N_{d_i/d_{i,j}}(a_i) = {N}_{d_j/d_{i,j}}(a_j), \quad 1 \leq i, j \leq k,$$ where $d_{i,j}=\gcd(d_i,d_j).$ 
    Moreover, in this case, we have $$\#\{\alpha\in\fqn^*\colon N_{n/d_i}(\alpha)=a_i, 1\leq i\leq k\}=\gcd\left(\frac{q^n-1}{q^{d_1}-1},\ldots,\frac{q^n-1}{q^{d_k}-1}\right).$$
\end{lemma}

\begin{proof}
    Let $\theta\in\fqn^*$ be a primitive element. Hence, for each $\alpha\in \F_{q^n}^*$, there exists a unique $0\leq s<q^n-1$  such that $\alpha=\theta^s$.
    Since every element in $\F_{q^{d_i}}^*$ can be seen as the norm of an element in $\fqn^*$, there exists an integer $0\leq t_i<q^n-1$ such that $$(\theta^{t_i})^\frac{q^n-1}{q^{d_i}-1}=a_i.$$
Therefore, $N_{n/d_i}(\alpha)=a_i$ if and only if 
$\label{eqn1}
    (\theta^{s})^\frac{q^n-1}{q^{d_i}-1}=(\theta^{t_i})^\frac{q^n-1}{q^{d_i}-1}.$
As $\theta$ is a primitive element, the last equality can be rewritten as the modular equation 
$$\label{eqn2}
    \frac{q^n-1}{q^{d_i}-1}\cdot s\equiv \frac{q^n-1}{q^{d_i}-1}\cdot t_i \pmod{q^n-1},$$
which is equivalent to $s\equiv t_i\pmod {q^{d_i}-1}$. As a consequence, we conclude that there exists $\alpha\in \F_{q^n}^*$ with $N_{n/d_i}(\alpha)=a_i$ for every $1\leq i\leq k$ if and only if the modular system
$$s\equiv t_i\pmod{q^{d_i}-1},\, 1\le i\le k,$$
has a solution. By the Chinese Remainder Theorem, the previous system has a solution if and only if 
    $t_i\equiv t_j \pmod{\gcd(q^{d_i}-1,q^{d_j}-1)}$
 or equivalently
\begin{equation}\label{eqn3}
t_i\equiv t_j \pmod {q^{d_{i,j}}-1}.   
\end{equation}
In the affirmative case, this system has exactly $$\frac{q^n-1}{\lcm(q^{d_1}-1,\ldots,q^{d_k}-1)}=\gcd\left(\frac{q^n-1}{q^{d_1}-1},\ldots,\frac{q^n-1}{q^{d_k}-1}\right)$$  incongruent solutions modulo $q^n-1$. Observe that since $\theta$ is a primitive element, such incongruent solutions will be in one-to-one correspondence with the elements $\alpha$ with prescribed norms. Hence it suffices to prove that Eq.~\eqref{eqn3} holds. 
Such equation is equivalent to $$\frac{q^n-1}{q^{d_{i,j}}-1}t_i\equiv \frac{q^n-1}{q^{d_{i,j}}-1}t_j \pmod {q^{n}-1}.$$
The latter holds since, from hypothesis, $N_{d_i/d_{i, j}}(a_i)=N_{d_j/d_{i, j}}(a_j)$ and then 
$$\theta^{\frac{q^n-1}{q^{d_{i,j}}-1}t_i}=N_{d_i/d_{i, j}}(a_i)=N_{d_j/d_{i, j}}(a_j)=\theta^{\frac{q^n-1}{q^{d_{i,j}}-1}t_j}.$$



\end{proof}

\subsection{The $\F_q$-order}
In this section we briefly discuss the $\F_q[x]$-module structure of finite fields that is induced by Frobenius maps.

\begin{definition}
    For $f\in\fq[x]$ with $f=\sum_{i=0}^ka_ix^i$, set $L_f(x)=\sum_{i=0}^ka_ix^{q^i}$
\end{definition}
Given $\alpha\in\fqn$, we set $f\circ\alpha:=L_f(\alpha).$ This defines an $\fq[x]$-module structure on $\F_{q^n}$. In other words, for every $f, g\in \F_q[x]$ and $\alpha, \beta\in \F_{q^n}$, the following hold:
\begin{enumerate}
    \item $f\circ (\alpha+\beta)=f\circ \alpha+f\circ \beta$;
    \item $1\circ \alpha=\alpha$;
    \item $(f+g)\circ \alpha=f\circ \alpha +g\circ \alpha$;
    \item $(fg)\circ \alpha=f\circ (g\circ \alpha)$.
\end{enumerate}
The proof of these identities follows by direct calculations so we omit them. Observe that for $\alpha\in \fqn$ we have $\alpha^{q^n}-\alpha=0$ and so $(x^n-1)\circ \alpha=0$. In particular, if $h\in \F_q[x]$ is the nonzero polynomial of least degree such that $h\circ \alpha=0$,  we easily conclude that $h$ must be a divisor of $x^n-1$. We arrive at the following definition.

\begin{definition}
    Let $\alpha\in\fqn$. The $\fq$-order of $\alpha$ is the monic divisor $g\in \F_q[x]$ of $x^n-1$ of least degree satisfying $g\circ \alpha=0$.
\end{definition}
Recall that an element $\alpha\in \F_{q^n}$ is normal over $\F_{q^n}$ if the set $\{\beta, \beta^q, \ldots, \beta^{q^{n-1}}\}$ is an $\F_q$-basis for $\F_{q^n}$. Equivalently, $\beta$ is a generator of $\F_{q^n}$ with the $\F_q[x]$-module structure previously defined. In particular, $\alpha$ is a normal element of $\fqn$ if and only if every $\beta\in\fqn$ can be written as $L_f(\alpha)$ for some $f\in\fq[x]$ with $\deg(f)<n$. Equivalently, $\alpha$ is normal if and only if its $\F_q$-order equals $x^n-1$. 

The existence of normal elements for every finite extension of finite fields is known. In particular, $\F_{q^n}$ is always a cyclic $\F_q[x]$-module. The following theorem provides a GCD criterion on the normality of elements in finite fields.

\begin{theorem}\cite[Theorem 2.39]{ff}\label{gcdnormal}
For $\alpha \in \mathbb{F}_{q^n}$, $\alpha$ is  normal over $\mathbb{F}_q$ if and only if the polynomials $x^n - 1$ and $\alpha x^{n-1} + \alpha^q x^{n-2} + \cdots + \alpha^{q^{n-1}}$ in $\mathbb{F}_{q^n}[x]$ are relatively prime.
\end{theorem}

\subsection{Characters and a characteristic function}
An additive character of $\fqn$ is a group homomorphism from the additive group $\fqn$ to the set of complex roots of unity. If $p$ is the characteristic of $\fq$ with $q=p^s$ and $\tr:\fqn\to\fp$ is the absolute trace function, that is, $\tr(\alpha)=\sum_{i=0}^{ns-1}\alpha^{p^i}$, the canonical additive character of $\fqn$ is the function
$$\psi_1(\alpha)=e^{\frac{2\pi i\tr(\alpha)}{p}}.$$ Every additive character of $\fqn$ can be written as $\psi_c(\alpha)=\psi_1(c\alpha)$ for some $c\in\fqn$. Moreover, the set $\widehat{\fqn}$ of additive characters of $\fqn$ is a group isomorphic to the additive group of $\fqn$. The neutral element of such group is the {\em trivial additive character} $\psi_0$, which maps every element of $\F_{q^n}$ to $1\in \mathbb C$.

\begin{definition}
    For a positive integer $n$, let $\mathbbm{1}_N$ be the characteristic function for normal elements in $\fqn$, that is, given $\alpha\in\fqn$, 
$$\mathbbm{1}_N(\alpha) =
\begin{cases} 
1 & \text{if } \alpha \text{ is normal}, \\
0 & \text{otherwise}.
\end{cases}$$
\end{definition}
We can express the characteristic function of normal elements in terms of additive characters. Before presenting the formula, we need to introduce some arithmetic functions over $\F_q[x]$.
\begin{definition}
    Given a monic polynomial $f\in\fq[x]$, we define the following quantities:
\begin{enumerate}[(i)]
    \item $\Phi_q(f) = 
\left|\left(
\frac{\mathbb{F}_q[x]}{\langle f \rangle}
\right)^*\right|$, that is, $\Phi_q(f)$ is the number of polynomials $g\in \F_q[x]$ of degree $<\deg(f)$ such that $\gcd(f, g)=1$. 
\item $\mu_q(f) = 0$ if $f$ is not square-free. Otherwise, $\mu_q(f) = (-1)^r$ where $r$ is the number of distinct irreducible divisors of $f$, defined over $\mathbb{F}_q$.
\item $W_q(f)$ is the number of squarefree monic divisors of $f$, defined over $\F_q$. 
\end{enumerate}
\end{definition}

We also can see $\widehat{\fqn}$ as an $\fq[x]$-module under the action $f\circ \psi:=\psi(L_f(x))$, where $f\in\fq[x]$ and $\psi\in\widehat\fqn$. We define the $\fq$-order of an additive character as follows.

\begin{definition}
    Let $\psi\in\widehat\fqn$. The $\fq$-order of $\psi$ is the monic divisor $h\in \F_q[x]$ of $x^n-1$ of least degree such that $\psi\circ h$ is the trivial character $\psi_0$. Moreover, $\Lambda_f$ stands for the set of additive characters $\psi$ of $\fqn$ whose $\fq$-order is $f$.
\end{definition}

In particular, the only additive character of $\F_q$-order $1\in\fq[x]$ is the trivial character.

\begin{remark}\label{rmk1}
If $f\in \F_q[x]$ is a monic divisor of $x^n-1$, then  $\# \Lambda_f=\Phi_q(f)$.
\end{remark}
The following lemma provides a character sum expression for the indicator function of the set of normal elements. 

\begin{lemma}\label{norma}
    For a positive integer $n$ and $\alpha$ in $\fqn$, we have that
    $$\mathbbm{1}_N(\alpha)=\frac{\Phi_q(x^n-1)}{q^n}\sum_{f\mid x^n-1}\frac{\mu_q(f)}{\Phi_q(f)}\sum_{\psi\in\Lambda_f}\psi(\alpha),$$
 where the outer sum is over the monic divisors of $x^n-1$, defined over $\F_q$.    
\end{lemma}
For a proof of Lemma \ref{norma} and Remark \ref{rmk1}, see Section 5.2 of \cite{epkn}.

\subsection{Estimates}
    


Here we provide estimates for some character sums and bounds on the function $W$ at polynomials of the form $x^n-1$. We start with the following result.

\begin{lemma}\cite[Corollary 1]{eiff}
\label{somacaracter}
    If $A, B \subseteq \fqn$ and ${\psi}$ is a nontrivial additive character
of $\fqn$, then we have
$$
\left| \sum_{x \in A} \sum_{y \in B} {\psi}(xy) \right| \leq (\#A \#B q^n)^{1/2}.
$$
\end{lemma}
We obtain the following corollary.

\begin{corollary}\label{cotacoset}
    Let $G$ be a subgroup of the multiplicative group $\fqn^*$ and $H$ a coset of $G$ in $\fqn^*$. If $\psi$ is a nontrivial additive character of $\fqn$, we have 
    $$\left|\sum_{\alpha\in H}\psi(\alpha)\right|\leq q^{n/2}.$$
\end{corollary}
\begin{proof}
    For each $\alpha\in H$, $$\sum_{y\in G}\psi(\alpha y)=
    \sum_{x\in H}\psi(x).$$   
    Taking $A=G$ and $B=H$ in Lemma \ref{somacaracter} and noticing that $\# G=\# H$, we get 

    $$
    \#H\cdot \left|  \sum_{x \in H} {\psi}(x) \right| =\left| \sum_{x \in H} \sum_{y \in G} {\psi}(xy) \right| \leq (\#H^2 q^n)^{1/2}
    $$
    and the result follows.
\end{proof}

In the following lemma we estimate the number of normal elements in a set where nontrivial additive character sums are uniformly bounded.

\begin{lemma}\label{normaiscoset}
    Let $M>0$ be a real number and let $S\subseteq \fq^*$ be a set such that, for every nontrivial additive character $\psi$, we have the inequality
    $$\left|\sum_{\alpha\in S}\psi(\alpha)\right|\leq M.$$ 
    If $N(S)$ is the number of normal elements $\alpha\in\fqn$ in $S$, then

    $$\frac{N(S)\cdot q^n}{\Phi_q(x^n-1)}>\#S-W_q(x^n-1)M.$$
\end{lemma}
\begin{proof}
It follows by the definition that 
$$N(S)=\sum_{\alpha\in S}\mathbbm{1}_N(\alpha).$$
By Lemma \ref{norma},
\begin{align}
    \frac{N(S)\cdot q^n}{\Phi_q(x^n-1)}&=\sum_{\alpha\in S}\sum_{f\mid x^n-1}\frac{\mu_q(f)}{\Phi_q(f)}\sum_{\psi\in\Lambda_f}\psi(\alpha)\nonumber\\
    &=\sum_{\alpha\in S}\psi_0(\alpha)+\underbrace{\sum_{\alpha\in S}\sum_{ f\mid x^n-1\atop f\ne 1}\frac{\mu_q(f)}{\Phi_q(f)}\sum_{\psi\in\Lambda_f}\psi(\alpha)}_{T}.
    \label{eq3}
\end{align}
Applying the triangular inequality on $T$ and reordering the summands, we obtain
\begin{align*}
    |T|&=\left|\sum_{ f\mid x^n-1\atop f\ne 1}\frac{\mu_q(f)}{\Phi_q(f)}\sum_{\psi\in\Lambda_f}\sum_{\alpha\in S}\psi(\alpha)\right|\\
    &\leq \sum_{f\mid x^n-1\atop f\ne 1}\frac{|\mu_q(f)|}{\Phi_q(f)}\sum_{\psi\in\Lambda_f}\left|\sum_{\alpha\in S}\psi(\alpha)\right|.
\end{align*}
From hypothesis,  $\left|\sum_{\alpha\in S}\psi(\alpha)\right|\leq M$ and, from Remark~\ref{rmk1}, we have $\#\Lambda_f=\Phi_q(f)$. Moreover,  $|\mu_q(f)|$ equals $1$ if $f$ is squarefree, and $0$ otherwise. Therefore 
we obtain

    \begin{equation}\label{eq4}
        |T|\leq \sum_{f\mid x^n-1\atop  f\ne 1\,\,   \text{and squarefree}}\frac{1}{\Phi_q(f)}\Phi_q(f)M=(W_q(x^n-1)-1)M.
    \end{equation}
Adding together Eq.~\eqref{eq3} and Ineq.~\eqref{eq4}, and applying the triangular inequality we obtain

\begin{align*}
    \frac{N(S)\cdot q^n}{\Phi_q(x^n-1)}&\geq\sum_{a\in S}\psi_0(a)-|T|.\\
    &>\#S-W_q(x^n-1)M.
\end{align*}
\end{proof}
Finally, we provide bounds for the quantity $W_q(x^n-1)$.
\begin{remark}\label{rmk2}
    For a polynomial $f\in\fq[x]$, recall that $W_q(f)$ denotes the number of squarefree divisors of $f$ in $\fq[x]$. In particular, $W_q(f)\le 2^{\deg(f)}$ and then we get the trivial bound $W_q(x^n-1)\le 2^n$. 
\end{remark}
    
    We introduce a useful result that provides a nontrivial bound for $W_q(x^n-1)$.

\begin{lemma}\cite[Lemma 3.7]{ep2n}\label{cotaw}
For $q$ a prime power, there exist $a, b \in \mathbb{N}$ such that
\[
W_q(x^n - 1) \leq 2^\frac{n + a}{b}.
\]
For $q \geq 29$, we have $a = 0$ and $b = 1$, for $7 \leq q \leq 27$ we have $a = q - 1$ and $b = 2$, and for small values of $q$ we may use the following values of $a$ and $b$:
\begin{itemize}
    \item $q=2, a=14, b=5;$
    \item $q=3, a=20 ,b=4;$
    \item $q=4, a=12 ,b=3;$
    \item $q=5, a=18 ,b=3.$
\end{itemize}
\end{lemma}

\section{Proof of main results}
In this section we provide the proof of our main results. We apply different techniques based on characteristic functions and nontrivial bounds for additive character sums. In the proof of Theorem~\ref{thm1}, after employing our methods, we are left with some finite number of inconclusive cases. For those, we search for the existence of our elements of interest by computational methods. In Appendix~\ref{ap1} we present a lemma in which our code is based.

\subsection{Proof of Theorem \ref{thm1}}
For a positive integer $n$ and $a\in\fq$, set $N_a=\{\alpha\in\fqn\colon N_{n/1}(\alpha)=a\}$. It is direct to check that $N_a$ is a coset of $N_1$ in $\fqn^*$. Therefore, by Corollary \ref{cotacoset} and Lemma \ref{normaiscoset}, $$\frac{N(N_a)\cdot q^n}{\Phi_q(x^n-1)}>\#N_a-W_q(x^n-1)q^{n/2}.$$
Since $\#N_a=\#N_1=\frac{q^n-1}{q-1}$, as it is the number of $\alpha\in\fqn$ such that $\alpha^\frac{q^n-1}{q-1}=1$, there exists a normal element $\alpha\in\fqn$ with $N_{n/1}$ provided that \begin{equation}\label{eq5}
    \frac{q^n-1}{q-1}>W_q(x^n-1)q^{n/2}.
\end{equation}
We proceed by computing the pairs $(q,n)$ satisfying Ineq.~\eqref{eq5}. By Remark \ref{rmk2} and Lemma \ref{cotaw} we can replace Ineq.~\eqref{eq5} by
\begin{equation}\label{eq6}
    \frac{q^n-1}{q-1}>2^nq^{n/2}.
\end{equation}
and 
\begin{equation}\label{eq7}
    \frac{q^n-1}{q-1}>2^{\frac{n+a}{b}}q^{n/2}.
\end{equation}
with $a,b$ as in Lemma \ref{cotaw}. It is important to notice that if $(q_0,n_0)$ satisfies Ineq.~\eqref{eq5}, \eqref{eq6} or $\eqref{eq7}$, then every pair $(q,n)$ with $q\geq q_0$ and $n\geq n_0$ also satisfies the same inequality. By a direct computation, we obtain that the result holds if
\begin{itemize}
    \item $3\leq n\leq 7, q\geq 64$;
    \item $n\geq 7$ and $q\ge 2$.
\end{itemize}

We analyze the remaining cases separately. For $3\le n\le 7$ and $q<64$, 
we use a SageMath program and check the existence of a normal element with one prescribed norm, obtaining positive answer to all the cases. For more details on the procedure, see Appendix~\ref{ap1}. 
Now for $n=2$, we recall that $\alpha\in\F_{q^2}$ is normal if and only if the $\fq$-order of $\alpha$ is $x^2-1$. Suppose that $\beta\in\F_{q^2}$ is not a normal element and $\beta^{q+1}=N_{2/1}(\beta)=b$, for some $b\in\fq^*$. In particular, $\beta\ne 0$ and so its $\F_q$-order equals $x+1$ or $x-1$, i.e., $\beta^{q}=\pm \beta$. This entails that $\beta$ satisfies 
\begin{center}
   \systeme{\beta^{q+1}=b\text{,},\beta^q=\pm \beta.} 
\end{center} We must have that $\beta^2=b$ or $-\beta^2=b$. Therefore, we can only have up to $4$ non normal elements $\beta\in \F_{q^2}$ with $N_{2/1}(\beta)=b$. Since we have $q+1$ elements with norm $b$, there always exists a normal element $\alpha$ in $\F_{q^2}$ with $N_{2/1}(\alpha)=b$ provided $q+1>4$. For $q=2, 3$, every normal element in $\F_{q^2}$ has norm $1$ over $\F_q$. 
\qed

\subsection{Proof of Theorem \ref{thm2}}

Let $n > 1$ be an integer that is not a prime power, $1 < k < \delta(n)$, $\mathcal{D} =
(d_1, \ldots, d_k) \in \Gamma_k(n)$ and $\mathcal{A} = (a_1, \ldots, a_k) \in \F^*_\mathcal{D}$ a $\mathcal{D}$-admissible tuple. Set $N_\mathcal{A}=\{\alpha\in\fqn^*\colon N_{n/d_i}(\alpha)=a_i, 1\leq i\leq k\}$. 
For $\mathcal{A}_1=\{1,\ldots,1\}\in\F^*_{\mathcal{D}}$, $N_\mathcal{A}$ is a coset of $N_{\mathcal{A}_1}$ in $\fqn^*$ and by Lemma \ref{normaiscoset} 

$$\frac{N(N_\mathcal{A})\cdot q^n}{\Phi_q(x^n-1)}>\#N_\mathcal{A}-W_q(x^n-1)q^{n/2}.$$

By Lemma \ref{lema1}, $\#N_\mathcal{A}=\gcd\left(\frac{q^n-1}{q^{d_1}-1},\ldots,\frac{q^n-1}{q^{d_k}-1}\right)$ and the result follows.  \qed

\subsection{Proof of Theorem \ref{thm3}}
Suppose that $\lcm(d_1,\ldots, d_k)=e<n.$ Since $k>1$, we have $e\geq 6$. From hypothesis, the tuple $(a_1, \ldots, a_k)$ is $\mathcal D$-admissible, therefore, Lemma~\ref{lema1} guarantees the existence of an element $\beta\in \F_{q^{e}}^*$ such that $N_{e/d_i}(\beta)=a_i$. Finally, since $e\ge 6$ we have $q^e>3$, hence Theorem~\ref{thm1} entails that there exists a normal element $\alpha\in \F_{q^n}$ such that $N_{n/e}(\alpha)=\beta$. We then obtain
$$N_{n/d_i}(\alpha)=N_{e/{d_i}}(N_{{n}/{e}}(\alpha))=N_{e/d_i}(\beta)=a_i,$$
for every $1\le i\le k$ and the result follows.  \qed

\subsection{Proof of Theorem \ref{thm4}}
 Let $n > 1$ be an integer that is not a prime power, $1 < k < \delta(n)$, $\mathcal{D} =
(d_1, \ldots, d_k) \in \Gamma_k(n)$ with $\gcd(d_i,d_j)=1$, $n=d_1\cdots d_k$ and let $\mathcal{A} = (a_1, \ldots, a_k) \in \F^*_\mathcal{D}$ be a $\mathcal{D}$-admissible tuple. Recall that $d_1<\cdots<d_k<n$. We start showing that 
\begin{equation}\label{eq12}
    \gcd\left(\frac{q^n-1}{q^{d_1}-1},\ldots,\frac{q^n-1}{q^{d_k}-1}\right)\geq \frac{q^n}{q^{d_1+\ldots+d_k}}.
\end{equation}
Observe that $$\gcd\left(\frac{q^n-1}{q^{d_1}-1},\ldots,\frac{q^n-1}{q^{d_k}-1}\right)=\frac{q^n-1}{\lcm(q^{d_1}-1,\ldots,q^{d_k}-1)}=\frac{(q^n-1)(q-1)^{k-1}}{\prod_{i=1}^k(q^{d_i}-1)}.$$
Since $d_i<n$, for every $1\leq i\leq k$, it follows that $1+\frac{1}{q^{d_i}-1}\geq 1+\frac{1}{q^{n}-1}>1.$ Therefore 
$$\prod_{i=1}^k\left(1+\frac{1}{q^{d_i}-1}\right)=\prod_{i=1}^k\left(\frac{q^{d_i}}{q^{d_i}-1}\right)\geq \left(\frac{q^n}{q^n-1}\right)^k>\frac{q^n}{q^{n}-1}$$ and Ineq.\eqref{eq12} follows.
From Theorem \ref{thm2}, we just need to check the inequality
$$q^{n/2-(d_1+\cdots+d_k)}\geq W_q(x^n-1)$$
for the cases described in the statement of Theorem~\ref{thm4}. We prove items 1 and 2 in the statement of Theorem \ref{thm4} separately.
\subsubsection{The case $k \geq 3$} We separate the cases $k\geq 4$ and $k=3$. 
\begin{itemize}
    \item $k\geq 4$: We have that  $n\geq 2\cdot3\cdot5\cdot7^{k-3}\geq 210$. Moreover, we observe that $\frac{n}{d_i}=\prod_{j\ne i}d_j\ge k!$ since the $d_j$'s are at least $2$ and pairwise relatively prime.
  Since $k!\ge 6k$ for every $k\ge 4$, we obtain  
    $$d_i=\frac{n}{\prod_{j\neq i}d_j}\leq \frac{n}{6k}.$$ 
    Therefore, $q^{n/2-(d_1+\cdots+d_k)}\geq q^{n/2-k n/6k}=q^{n/3}$. Using the bounds for $W_q(x^n-1)$, there exists a normal element on $\fqn$ with prescribed norms provided that 

\begin{equation}\label{eq13}
    q^{n/3}\geq 2^{\frac{n+a}{b}},
\end{equation} where $a,b$ are as in Lemma \ref{cotaw}. For $a=0$ and $b=1$, Ineq.~\eqref{eq13} holds for every $q>8$ and every $n$. For the range $2\le q\le 8$, we use values for $a$ and $b$ in Lemma~\ref{cotaw} and conclude that Ineq.~\eqref{eq13} holds whenever $n\ge 60$. Since $n\ge 210$, we completed the case $k\geq 4$.


\item $k=3$: Since $\gcd(d_i, d_j)=1$, we have $n\ge 60$ unless $n=30=2\cdot 3\cdot 5$ or $n=42=2\cdot 3\cdot 7$. Assume that $n\ne 30, 42$, hence $n\ge 60$. 
If $d_1=2, d_2=3$ and $d_3=n/6$, then $d_1+d_2+d_3= 2+3+n/6.$ 
Otherwise, we note that $d_1d_2\geq 10$, $d_3\leq n/10$ and $n\ge 2\cdot 5\cdot 7=70$.  We always have
$d_1\leq \sqrt[3]{n}$ and $d_2\leq\sqrt{n/2}$. Therefore, 
$d_1+d_2+d_3\le \sqrt[3]{n}+\sqrt{n/2}+n/10$.
As $2+3+n/6\le n/4$ for $n\ge 60$ and $\sqrt[3]{n}+\sqrt{n/2}+n/10\le n/4$ for $n\ge 70$, we conclude that $d_1+d_2+d_3\le n/4$, hence $n/2-d_1-d_2-d_3\ge n/4$. 
As in the previous case, the result holds provided that 
\begin{equation}\label{eq14} q^{n/4}\geq 2^\frac{n+a}{b}.
\end{equation}
As before, we employ Lemma \ref{cotaw}.
For $a=0$ and $b=1$, the result holds for every $q\geq 17$ and every $n\geq 60$. For $q\leq 16$, we use the values for $a$ and $b$ in Lemma~\ref{cotaw} to conclude that the result also holds for $n\geq 60$. It remains to consider the cases $n=30$ and $n=42$. For these values, Theorem \ref{thm2} entails that the result holds if 
\begin{equation}\label{eq15}
\frac{(q^n-1)(q-1)}{(q^{d_1}-1)\cdot(q^{d_2}-1)\cdot(q^{d_3}-1)}\geq W_q(x^n-1)q^{n/2}
\end{equation}
with $d_1=2, d_2=3$, $d_3=5$ and  $d_1=2, d_2=3$, $d_3=7$, respectively. Using the bound $W_q(x^n-1)\leq 2^n$, we conclude the result for $q\geq 23$. For $2\leq q\leq 19$, we use SageMath to compute the explicit value of $W_q(x^{30}-1)$ and $W_q(x^{42}-1)$. We also check that the result holds for every $q\geq 2$, and this concludes the case $k=3$.
\end{itemize}

\subsubsection{The case $k=2$}
We split the proof into cases.
\begin{itemize}
    \item If $8\leq d_1<d_2$, then $d_i\leq n/8$ and $d_1+d_2\leq n/4$. The result follows by the same argument employed in the case $k=3$. 
    \item If $7\leq d_1<d_2$, then $d_2\leq n/7$ and $d_1\leq n/8$. Therefore, it suffices to check that $$q^{n/2-n/7-n/8}=q^{\frac{13}{56}n}\geq 2^{\frac{n+a}{b}}.$$ Again, we employ Lemma \ref{cotaw}. For $a=0, b=1$, the inequality is true for $q>19$. We then use the values for $a$ and $b$ on Lemma \ref{cotaw} to conclude that the result holds for $2\leq q\leq 19$. 
    \item Let $d=d_1\in \{3, 4, 5, 6\}$, hence $d_1+d_2=d+n/d$.
    Theorem \ref{thm2} entails that the result holds if
    $$\frac{(q^n-1)(q-1)}{(q^{d_1}-1)(q^{d_2}-1)}>q^{n/2}2^{n}.$$
We observe that the LHS of the previous inequality is bounded below by 
$$\frac{q-1}{q}q^{n-d_1-d_2+1}=\frac{q-1}{q}q^{n-d-n/d+1}.$$
Therefore, the results holds if $q$ is sufficiently large, provided that 
$$n-d-n/d+1>n/2\iff \frac{1}{2}>\frac{1}{d_1}+\frac{1}{d_2}-\frac{1}{d_1d_2}.$$
It is direct to verify that the latter holds if $(d_1, d_2)\ne (3, 4)$. \qed

\end{itemize}


 

\appendix
\renewcommand{\thesection}{\Alph{section}} 
\makeatletter
\renewcommand\@seccntformat[1]{\appendixname\ \csname the#1\endcsname.\hspace{0.5em}}
\makeatother
\section{Normal elements with one prescribe norm}\label{ap1}
The following result is straightforward.
\begin{lemma}\label{lem:ap}
    Let $n$ be a positive integer. If $\mathcal N$ is the set of normal elements of $\fqn$ and $$x^{q-1}-1=\prod_{\alpha\in\fq^*}(x-\alpha)\bigm| \prod_{\beta\in \mathcal N}\left(x-\beta^{\frac{q^n-1}{q-1}}\right),$$
    then for every $a\in\fq^*$ there exists a normal element $\beta$ with $N_{n/1}(\beta)=a$.
\end{lemma}
Moroever, Theorem \ref{gcdnormal} entails that  $$\mathcal N=\left\{\alpha\in\fqn^*\colon \gcd\left(x^n-1, \sum_{i=0}^{n-1}\alpha^{q^i}x^{n-1-i}\right)=1\right\}$$ is the set of normal elements of $\fqn$. Hence, given $(q, n)$ we only need to compute the set $\mathcal N$, and check if the divisibility condition in Lemma~\ref{lem:ap} holds. This is easily implemented in softwares like SageMath.

\end{document}